\documentclass{base_class}

\title[]{Octonion-Valued Forms and the Canonical 8-Form on Riemannian Manifolds with a $Spin(9)$-Structure} 

\author{Jan Kotrbat{\'y}}\thanks{Supported by DFG Grant WA3510/1-1.}
 
\email{jan.kotrbaty@uni-jena.de}

\begin{document}

\maketitle

\begin{abstract}
It is well known that there is a unique $Spin(9)$-invariant 8-form on the octonionic plane that naturally yields a canonical differential 8-form on any Riemannian manifold with a weak $Spin(9)$-structure. Over the decades, this invariant has been studied extensively and described in several equivalent ways. In the present article, a new explicit algebraic formula for the $Spin(9)$-invariant 8-form is given. The approach we use generalises the standard expression of the K\"{a}hler 2-form. Namely, the invariant 8-form is constructed only from the two octonion-valued coordinate 1-forms on the octonionic plane. For completeness, analogous expressions for the Kraines form, the Cayley calibration and the associative calibration are also presented.
\end{abstract}

\section{Introduction}

One of the most common features of the reals as well as of the complex numbers is the compatibility of their product with the norm. 
Famously, besides $\RR$ and $\CC$, this is an exclusive property of only two other spaces: the four-dimensional algebra $\HH$ of {\it quaternions} and the eight-dimensional algebra $\OO$ of {\it octonions} (sometimes also referred to as {\it Cayley numbers} or {\it octaves}).

Since this striking result of Hurwitz \cite{Hurwitz1922} was published, the fact that precisely four normed division algebras exist has turned out to be extremely generic as it has been observed to underlie a wide variety of other classification theorems. Remember, for instance, simple Lie algebras: the three classical series 
correspond to $\RR$, $\CC$ and $\HH$ (see e.g. \cite{knapp}, \S I.8) while the five exceptions 
are closely tied to $\OO$ (see \S4 of the excellent survey \cite{baez}). Formally real Jordan algebras are categorised similarly, see \cite{jordan}. Or, as shown by Adams \cite{adams1960}, there exist precisely four Hopf fibrations; see \S\ref{ss:spin9} for construction of the `octonionic' one, the others are obtained simply replacing $\OO$ by the other normed division algebras (see \cite{baez,hopf}). And the list continues.


In the forties, Borel \cite{borel} and Montgomery with Samelson \cite{montgomery} classified compact connected Lie groups acting transitively and effectively on a unit sphere. These are
\begin{align}
\begin{gathered}
\label{eq:BMS}
SO(n)=SO(\RR^n);\\
U(n),SU(n)\subset SO(\CC^n);\\
Sp(n),Sp(n)U(1),Sp(n)Sp(1)\subset SO(\HH^n);\\
G_2\subset SO(\Imag\OO); \, Spin(7)\subset SO(\OO);\, Spin(9)\subset SO(\OO^2);
\end{gathered}
\end{align}
in each case the action on the unit sphere in $V$ is inherited from $SO(V)$. The reason to write $\HH^n$ instead of $\RR^{4n}$, $\Imag\OO$ (octonions with zero real part, see \S2.1) rather than $\RR^7$, etc., is that all the listed groups are naturally realised in terms of the respective normed-division-algebra structures (see \S2.2 for $Spin(9)$). A few years later, Berger's classification \cite{berger1955} of admissible holonomies of non-symmetric Riemannian manifolds resulted independently in almost the same list as \eqref{eq:BMS} lacking only the series $Sp(n)U(1)$. The conjecture that this was not a coincidence was then settled by Simons \cite{simons} (see also \cite{besse}, \S10).

Over the decades, all the holonomies from Berger's list have been realised (for a detailed exposition see \cite{salamon1989}), with one exception, though. It was proposed by Alekseevskij \cite{alekseevskii} and later proven by Brown and Gray \cite{BG} that in fact any manifold with holonomy $Spin(9)$ is isometric to either $\OO P^2\cong F_4/Spin(9)$ or to its non-compact dual $\OO H^2\cong F_{4(-20)}/Spin(9)$. In particular, such a manifold is always symmetric. Even though the significance of the Cayley planes $\OO P^2$ and $\OO H^2$ was undeniable for many reasons (let us mention, for instance, that $\OO P^2$ is an example of non-Desarguesian plane \cite{moufang}, see also \S3 of Reference \cite{baez} and \S3 of Reference \cite{besseclosed}), it may have seemed that the $Spin(9)$-structures were quantitatively rather poor. The latter discovery of Friedrich \cite{friedrich2001,friedrich2003} however shows that the Cayley planes constitute just one of the total of sixteen distinguished classes of 16-dimensional Riemannian manifolds that admit a {\it weak} $Spin(9)$-structure, i.e. a $Spin(9)$-reduction of the frame bundle.

The study of $Spin(9)$-structures is strongly motivated, inter alia, by modern physics. Let us give a brief resume. An octonionic description of quantum mechanics involving $\OO P^2$ was presented in Reference \cite{G1978}. The role of the Cayley planes in connection to supergravity was discussed in References \cite{deWit,G1984} and \cite{G1985}. The Lie group $Spin(9)$ is of serious interest in general M-theory \cite{babalic,BFSS,filev,sati2009,sati2011}, and it is also linked to Seiberg--Witten theory \cite{SWth} and to supersymmetric Yang--Mills theory \cite{YM2,YM3,YM1}.

\vspace{2ex}

There is yet another significant feature of the aforementioned contribution by Friedrich and this is the subject of the present article. It is a remarkable fact, first observed by Brown and Gray \cite{BG}, that the space $\bigwedge^8(\OO^2)^*$ contains a non-trivial $Spin(9)$-invariant element $\Psi$ that is unique up to a scalar factor and yields the canonical parallel 8-form on $\OO P^2$. Naturally, via the defining restriction of the frame bundle, $\Psi$ induces a canonical (in general not parallel) 8-form on any manifold with a given weak $Spin(9)$-structure. Now let us turn our attention to this fundamental invariant.

Brown and Gray \cite{BG} initially described the $Spin(9)$-invariant 8-form $\Psi$ as a certain, rather complicated Haar integral over the subgroup $Spin(8)$ of $Spin(9)$. Already at the same year Berger \cite{berger} gave the following elegant integral formula:
\begin{equation}
\label{eq:c1}
\Psi=c_1\int_{\OO P}\pi_\L^*\,\nu_\L \,d\L,
\end{equation}
where $c_1\in\RR$ is a constant, $\nu_\L$ is the volume form on the octonionic line $\L\cong\RR^8$, $\pi_\L\maps{\OO^2}{\L}$ is the projection and $d\L$ is the Haar measure on the octonionic projective line $\OO P$ (see \S\ref{ss:spin9} for a precise definition).

The first attempts towards an algebraic expression of $\Psi$ were due to Brada and P{\'e}caut-Tison \cite{brada}, and Abe and Matsubara \cite{abe}, respectively. Unfortunately, as explained in Reference \cite{lopez2010}, neither of the two approaches was completely correct. The combinatorial formula \cite{abe} was however corrected after some time (see \cite{private,PP}).

Later on, Castrill{\'o}n L{\'o}pez et al. \cite{lopez2010} expressed the canonical 8-form in terms of the generators $\I_{ij}$ of the Lie algebra $\mathfrak{spin}(9)$. Namely,
\begin{equation}
\label{eq:c2}
\Psi=c_2\sum_{i,j,k,l=0}^8 \omega_{ij}\wedge\omega_{ik}\wedge\omega_{jl}\wedge\omega_{kl},
\end{equation}
where $c_2\in\RR$ is a constant and $\omega_{ij}:=\ip{\Cdot}{\I_{ij}\Cdot}\in\bigwedge^2(\OO^2)^*$, $\ip{\Cdot}{\Cdot}$ is the inner product on $\OO^2$ (see also \S\ref{ss:spin9}).

Further interpretation was provided by Parton and Piccini \cite{PP} who used the computer algebra system {\tt Mathematica} to compute all 702 terms of $\Psi$ directly from \eqref{eq:c1} and proved further that the 8-form is proportional to the fourth coefficient of the characteristic polynomial of the matrix $(\omega_{ij})_{i,j=0}^8$. Very recently, Castrill{\'o}n L{\'o}pez et al. \cite{lopez2017} showed that this approach differs from \eqref{eq:c2} just from a combinatorial point of view in fact.

Let us mention that the spin module $\OO^2$ of $Spin(9)$ is not the only one among \eqref{eq:BMS} possessing a non-trivial canonical form (see e.g. \cite{besse}, Table 1 on p. 311). Besides $\Psi$, one has the K\"{a}hler form $\omega$, the Kraines form $\Omega$, the associative calibration $\phi$ and the Cayley calibration $\Phi$. Respectively, they live on $V=\CC^n,\HH^n,\Imag\OO,\OO$, are of degree 2, 4, 3 and 4, and are invariant under $U(n)$, $Sp(n)Sp(1)$, $G_2$ and $Spin(7)$. All these forms are stabilised in $GL(V)$ by precisely the respective groups they are invariant under (for the case of $\Psi$ see \cite{lopez2010}, Theorem 3.1). Recall that the canonical forms play a principal role from the point of view of calibrated geometries (see \cite{calibr}).

\vspace{2ex}

It is the aim of this article to present a new algebraic formula for the $Spin(9)$-invariant 8-form. To this end we introduce the concept of {\it octonion-valued forms} (see \S3 for precise definition). Although this formalism is very intuitive and simple, it sheds, we believe, new light on how tight the connection between $Spin(9)$ and the octonions really is, and illuminates the deep octonionic nature of $\Psi$. At this stage, let us mention a very recent paper by Grigorian \cite{grigorian2017} where, to the best of our knowledge, the term `octonion-valued form' appeared for the first time, referring to an $(\Imag\OO)$-valued {\it differential} form on an octonionic $G_2$-bundle. The current article, however, had been essentially finished when the presence of Grigorian's work was revealed to us and therefore the latter did not affect our choice of notation at all.

Our method generalises naturally the standard expression of the K\"{a}hler form $\omega$ on $\CC^n$. In terms of the complex coordinate 1-forms $dz_1,\dots,dz_n$ on $\CC^n$, this canonical $U(n)$-invariant 2-form is usually given by (see e.g. \cite{tubes}, p. 102)
\begin{align}
\label{eq:kaehler}
\omega=\frac{i}2\sum_{j=1}^ndz_j\wedge\b{dz_j}.
\end{align}

The following observation is then central to our approach. In spite of being a real form, $\omega$ is regarded as an element of a bigger (real) algebra, namely, $\CC\otimes\bigwedge^\bullet(\CC^n)^*$, equipped with the wedge product arising naturally on the tensor product of two algebras and with the involution extended from $\CC$. This allows one to compress the expression in real coordinates into the symmetric formula \eqref{eq:kaehler}. Is it possible to move further along the path $\RR-\CC-\HH-\OO$ in order to obtain analogous expressions also for the other canonical invariants, in particular for the 8-form $\Psi$? The answer is `yes', but one must be slightly more careful: neither $\HH$ nor $\OO$ is commutative and $\OO$ is not even associative!

First, consider the Kraines 4-form. Let $dw_1,\dots ,dw_n$ be the quaternionic coordinate 1-forms on $\HH^n$ and let $\Omega_{ij}:=dw_i\wedge\b{dw_j}$. Then the Kraines form, regarded as an element of $\HH\otimes\bigwedge^\bullet(\HH^n)^*$, can be written as follows:
\begin{align}
\label{eq:kraines}
\Omega=-\frac14\sum_{i,j=1}^n\Omega_{ij}\wedge\b{\Omega_{ij}}.
\end{align}
Although it is quite straightforward to transform the standard definition \cite{kraines} of $\Omega$ into \eqref{eq:kraines}, we have not found this in literature and therefore the formula \eqref{eq:kraines} is derived in \S4 below.

Second, even simpler are the cases of the (self-dual) Caley calibration $\Phi$ and the associative calibration $\phi$. Let $dx$ be the octonionic coordinate 1-form on $\OO$. Then it is easy to see that
\begin{align}
\label{eq:cayley}
\Phi=-\frac1{24}(dx\wedge\b{dx})\wedge(dx\wedge\b{dx}).
\end{align}
Similarly, if $dx$ is regarded as the coordinate 1-form on $\Imag\OO$, one has $\b{dx}=-dx$ and thus
\begin{align}
\label{eq:assoccal}
\phi=-\frac1{12}\left[(dx\wedge dx)\wedge dx+dx\wedge(dx\wedge dx)\right].
\end{align}
Notice that the brackets are necessary in \eqref{eq:cayley} and \eqref{eq:assoccal} since the wedge product on $\OO\otimes\bigwedge^\bullet(\OO)^*$ and $\OO\otimes\bigwedge^\bullet(\Imag\OO)^*$ is no longer associative.

Finally, and it forms the nucleus of our paper, we apply the analogous notion to the canonical $Spin(9)$-invariant 8-form $\Psi$ that we thereby regard as an element of the algebra $\OO\otimes\bigwedge^\bullet(\OO^2)^*$. Let $dx$ and $dy$ be the octonionic coordinate 1-forms on $\OO^2$ and let us denote
\begin{align*}
\Psi_{40}&:=((\b{dx}\wedge dx)\wedge \b{dx})\wedge dx,\\
\Psi_{31}&:=((\b{dy}\wedge dx)\wedge \b{dx})\wedge dx,\\
\Psi_{13}&:=((\b{dx}\wedge dy)\wedge \b{dy})\wedge dy,\\
\Psi_{04}&:=((\b{dy}\wedge dy)\wedge \b{dy})\wedge dy.
\end{align*}
Then our main result is as follows.
\begin{theorem}
\label{thm1}
The form
\begin{equation}
\label{eq:c4}
\Psi_8:=\Psi_{40}\wedge\b{\Psi_{40}}+4\,\Psi_{31}\wedge\b{\Psi_{31}}-5\left(\Psi_{31}\wedge\Psi_{13}+\b{\Psi_{13}}\wedge\b{\Psi_{31}}\right)+4\,\Psi_{13}\wedge\b{\Psi_{13}}+\Psi_{04}\wedge\b{\Psi_{04}}
\end{equation}
is a non-trivial real multiple of the Spin(9)-invariant 8-form $\Psi$ on $\OO^2$.
\end{theorem}
Notice that, unlike for the Kraines form $\Omega$, the aforementioned algebraic formulas for $\Psi$ are too complicated to be simply rewritten into \eqref{eq:c4}. Instead, the form $\Psi_8$ is constructed independently in several steps following the requirements of non-triviality and $Spin(9)$-invariance.

Let us emphasise two crucial advantages of our approach to the problem. First, the presented description of the form $\Psi$ allows us to verify the non-triviality and the invariance readily with very simple algebraic tools and to eliminate the role of combinatorics significantly. From this point of view, the language of octonion-valued forms we use seems to be a very natural one. Second, we are able to determine explicitly all the 702 terms $\Psi$ has in the standard basis. This explains the pattern Parton and Piccinni \cite{PP} observed in their Table 2. Notice, however, that all the computations we perform here are completely independent of any computer aid.

\vspace{2ex}

To conclude the introduction, let us briefly discuss some further aspects of our work. Namely, we would like to suggest two particular research directions that the introduced notion of octonion-valued forms may be possibly utilised in or at least related to.

First, a new dimension of the Borel--Montgomery--Samelson list \eqref{eq:BMS} was revealed at the turn of the millennium by the work of Alesker \cite{alesker2000}. He showed that for a compact subgroup $G\subset SO(n)$, the algebra of translation- and $G$-invariant continuous valuations on convex bodies in $\RR^n$ is finite-dimensional if and only if $G$ acts transitively and effectively on the sphere $S^{n-1}$. Alesker's discovery naturally launched systematic exploration of valuation algebras. The past two decades have witnessed a massive development of various algebraic structures on valuations as well as of their influence on integral geometry, pioneered by Alesker, Bernig, Fu and others (see e.g. \cite{alesker2001,alesker2003,alesker2004a,alesker2004,alesker2011,abs,AB2004,bernig2011,bf2006,bf2011,fu2006}). Nonetheless, the problem of the description of invariant valuations is still far from being solved completely, in particular the cases corresponding to the three series of symplectic groups as well as to $Spin(9)$ remain almost completely open (for partial, yet significant results see \cite{alesker2008,bernig2012,bs2014,bs2017,voide}). Our ongoing research suggests that it is natural and very convenient to apply the concept of octonion-valued forms to the case of $Spin(9)$-invariant valuations on $\OO^2$. Originally, this was our major motivation for the present work.


Second, Parton and Piccinni \cite{PP2015} very recently generalised, in a very smooth and beautiful way, their description of $\Psi$ to the canonical $Spin(10)$-invariant 8-form on $\CC^{16}\cong\CC\otimes\OO^2$. Piccinni \cite{piccinni2017} further prolonged this approach to (no longer canonical) 8-forms on $\HH^{16}$ and $\OO^{16}$ that are invariant under $Spin(12)$ and $Spin(16)$, respectively. This  naturally raises the question whether our method admits an analogical generalisation.

\vspace{2ex}

The present paper is organised as follows. First, necessary background is recalled involving the octonions and the Lie groups Spin(9) and Spin(8). Further, octonion-valued forms are defined as the elements of the algebra $\OO\otimes\bigwedge\!{}^\bullet V^*$ for a general finite-dimensional real vector space $V$, and their properties are discussed. In the third section, the formula \eqref{eq:kraines} for the Kraines form is derived. The fourth section is devoted to the proof of our main result, i.e. the 8-form $\Psi_8$ is defined through octonion-valued forms on $\OO^2$ and then proven to be non-trivial and $Spin(9)$-invariant. Finally, the explicit expression of $\Psi_8$ in the standard basis is given in the Appendix.

\subsection*{Notation}
We shall employ the following notation throughout. The superscript star will denote the adjoint mapping, the pullback or the dual vector space, depending on the context. $\linspan$ will always refer to a {\it real} linear hull. By $\id$ we shall always mean the identity mapping on $\OO\cong\RR^8$. $\diag(A,B)$ will stand for the diagonal two by two block matrix with $A$ and $B$ on the diagonal. Finally, for a G-module $S$, $S^G$ will be the subspace of $S$ consisting of $G$-invariant elements.

\subsection*{Acknowledgements}
I would like to thank my advisor Prof. Thomas Wannerer, who gave me the initial impulse to deal with this problem and supported my work by plenty of useful ideas and comments.

\section{Preliminaries}
\label{prelim}

\subsection{The Octonions}
\label{octonions}
Let us begin with a review of the biggest normed division algebra, the 8-dimensional algebra $\OO$ of {\it octonions}. As anticipated, $\OO$ is neither abelian nor associative. It is, nonetheless, still {\it alternative}, i.e. any subalgebra generated by two elements of $\OO$ is associative (see e.g. \cite{baez}, p. 149).

Working within the octonions, we shall adhere to the following conventions. At first, $\OO$ as a real vector space is just $\RR^8$ with the standard basis denoted by $\{1, e_1,\dots,e_7\}$. Then we define the involution as $\b1=1$ and $\b{e_j}=- e_j$. Regarding the algebra structure, the first basis element is the multiplicative unit; the rest of the product is determined by $e_i^2=-1$ and $e_ie_j=-e_j e_i$, $i\neq j$, together with the requirement of alternativity and the rule
\begin{equation}
\label{eq:Omul}
\left(e_{(1+i)\hspace{-1.5ex}\mod 7}\right)\!\left(e_{(2+i)\hspace{-1.5ex}\mod 7}\right)=e_{(4+i)\hspace{-1.5ex}\mod 7},\quad 1\leq i\leq 7.
\end{equation}
To illustrate this, for $i=5$, for instance, \eqref{eq:Omul} implies not only $e_6e_7=e_2$ but also
\begin{align*}
e_7e_2=e_7(e_6e_7)=-e_7(e_7e_6)=-(e_7e_7)e_6=e_6
\end{align*}
and similarly $e_2e_6=e_7$. The inner product on $\OO$ is given by $\ip{u}{v}=\Real (u\b v)$, where the {\it real-part operator} acts as $\Real(u)=\frac12(u+\b u)$, and the induced norm is denoted $\sqnorm{u}=\ip{u}{u}=u\b u$. Such defined inner product agrees with the standard Euclidean structure on $\RR^8$ and (thus) the considered basis is orthonormal with respect to it. 

Let $R_u\mape{x}{xu}$ and $L_u\mape{x}{ux}$ denote the {\it right} and {\it left}, respectively, {\it multiplication} by $u\in\OO$ in $\OO$. Then the following relations hold for any $u,v,w\in\OO$ (see Reference \cite{SW}, p. 26):
\begin{align}
\label{eq:RLuip}
\ip{R_u(v)}{w}&=\ip{v}{R_{\b u}(w)}\quad\text{and}\quad\ip{L_u(v)}{w}=\ip{v}{L_{\b u}(w)},\\
\label{eq:RuRv}
R_{\b u} R_{v}+R_{\b v} R_{u}&=L_{\b u} L_{v}+L_{\b v} L_{u}=2\ip{u}{v}\cdot\id.
\end{align}
Setting $v=u$ in \eqref{eq:RuRv} together with \eqref{eq:RLuip} implies $R_u,L_u\in O(8)$ iff $\norm{u}=1$. Moreover, since $S^7\subset\OO$ is connected and $R_1=L_1=\id$, all right and left multiplications lie in the identity component of $O(8)$ in fact, i.e. $R_u,L_u\in SO(8)$, provided $\norm{u}=1$. 

Further useful relations in $\OO$ are the so-called {\it Moufang identities} (see e.g. \cite{harvey}, p. 120):
\begin{align}
\label{eq:M1}
(uvu)w&=u(v(uw)),\\
\label{eq:M2}
w(uvu)&=((wu)v)u,\\
\label{eq:M3}
(uv)(wu)&=u(vw)u.
\end{align}
Let us emphasise that the brackets are in general necessary here due to non-associativity of $\OO$. On the other hand, one may omit them in expressions generated by at most two elements. 

\subsection{The Group Spin(9)}
\label{ss:spin9}

There are several equivalent ways to define the Lie group $Spin(9)$. Somewhat abstractly, one can say that it is the universal (two-fold) covering group of $SO(9)$. Let us present two more explicit definitions, both of them being related to the octonions.

First, $Spin(9)$ is the group of symmetries of the octonionic Hopf fibration $S^7\hookrightarrow S^{15}\rightarrow S^8$ (see Reference \cite{hopf}). By {\it symmetry} one means a rigid motion of the total space taking fibres to fibres. To describe the fibration, for any $a\in\b\OO:=\OO\cup\{\infty\}$ we define the {\it octonionic line} $\L_a\cong\OO$ as 
\begin{align*}
\L_a&:=\left\{\begin{pmatrix}u\\ua\end{pmatrix}\in\OO^2 ; u\in\OO \right\}, \text{ if }a\in\OO, \text{ and}\\
\L_\infty&:=\left\{\begin{pmatrix}0\\u\end{pmatrix}\in\OO^2 ; u\in\OO \right\}.
\end{align*}
Clearly, any two lines intersect only at the origin and $\bigcup_{a\in\b\OO}\L_a=\OO^2$. Then the projection from the total space $S^{15}=\left\{\begin{pmatrix}u_1\\u_2\end{pmatrix}\in\OO^2; \sqnorm{u_1}+\sqnorm{u_2}=1\right\}$ onto the base $\OO P=\{\L_a;a\in\b\OO\}\cong S^8$ just sends a unit vector from $\OO^2$ to the octonionic line it belongs. Clearly, the fibre over $\L_a\in\OO P$ equals to $S^{15}\cap\L_a\cong S^7$. In other words, one can say that $Spin(9)$ is the subgroup of $SO(16)$ that preserves the octonionic projective line $\OO P$ in the Grassmannian $\Grass_8(\OO^2)$.

Second, the same group can be realised in terms of Clifford algebras and spinors (see Reference \cite{harvey}, p. 288). In that case it turns out that $Spin(9)$ is generated by
\begin{equation*}
\left\{\begin{pmatrix}R_r&R_u\\R_{\b u}&-R_r\end{pmatrix}\in SO(16);r\in\RR,u\in\OO,r^2+\sqnorm{u}=1\right\}.
\end{equation*}
Related to this approach is the description of the Lie algebra $\mathfrak{spin}(9)=\mathfrak{so}(9)\subset\mathfrak{so}(16)$. Namely, consider the following nine elements of the above generating set:
\begin{equation*}
\I_j:=\begin{pmatrix}0&R_{e_j}\\ R_{\b {e_j}}& 0\end{pmatrix},0\leq j\leq 7,\quad\text{and}\quad\I_8:=\begin{pmatrix}\id&0\\ 0& -\id\end{pmatrix},
\end{equation*}
and denote $\I_{jk}:=\I_j\I_k$. The relations $\I_{jj}^2=-\id$ and $\I_{jk}=-\I_{kj}$, $j\neq k$, are easily verified, the second one using \eqref{eq:RuRv}. Therefore, whenever $j\neq k$, $\I_{jk}^{-1}=-\I_{jk}$ and so $\I_{jk}^*=-\I_{jk}$. Further, the set $\{\I_{jk};0\leq j<k\leq 8\}$ is linearly independent (see Reference \cite{PP}, Proposition 8) and its elements satisfy
\begin{equation}
\label{eq:com}
[\I_{jk},\I_{lm}] = 
\begin{cases}
0&\quad\text{if } \{j,k\}\cap\{l,m\}=\emptyset,\\
2\I_{km}&\quad\text{if } j=l,k\neq m,\\
-2\I_{kl}&\quad\text{if } j=m,k\neq l.
\end{cases}
\end{equation}
Therefore, the subspace $\linspan\{\I_{jk};0\leq j<k\leq 8\}$ is a $36$-dimensional subalgebra of $\mathfrak{so}(16)$, and since the corresponding one-parameter subgroups $g_{jk}(t):=\exp(t\I_{jk})=\cos(t)\id+\sin(t)\I_{jk}$ all belong to $Spin(9)$, it is in fact $\mathfrak{spin}(9)$.

\subsection{The Group Spin(8)}
\label{ss:spin8}
Correspondingly to $SO(8)\subset SO(9)$, the Lie group $Spin(9)$ contains the double cover $Spin(8)$ of $SO(8)$. It is discussed, e.g. in Reference \cite{harvey}, p. 278, that this Lie group can be realised, again with help of the octonions, as follows:
\begin{equation*}
Spin(8)=\left\{\begin{pmatrix}g_+&0\\0&g_-\end{pmatrix};g_+,g_-\in O(8), g_+(xy)=g_-(x)g_0(y),\text{ for all }x,y\in\OO\right\},
\end{equation*}
where the (irreducible) {\it vector representation} $\rho_0$ of $Spin(8)$ on $S_0=\OO$ is given by
\begin{equation*}
\rho_0\mape{\diag\left(g_+,g_-\right)}{g_0:=L_{\b{g_-(1)}}\circ g_+}.
\end{equation*}

As one may observe from \eqref{eq:com}, the $28$-dimensional subspace $\linspan\{\I_{jk};0\leq j<k\leq 7\}$ is a subalgebra of $\mathfrak{spin}(9)$. The corresponding one-parameter subgroups take the form
\begin{align}
\label{eq:gjk8}
g_{jk}(t)=\diag\left(R_{e_j}\circ R_{\b{ u_{jk}(t)}},R_{\b{e_j}}\circ R_{u_{jk}(t)}\right)\in Spin(8),
\end{align}
where we denoted $u_{jk}(t):=\cos(t)e_j+\sin(t)e_k\in\OO$, so they generate $Spin(8)\subset Spin(9)$.

The {\it positive} and {\it negative}, respectively, {\it spin representations} $\rho_\pm$ of $Spin(8)$ on $S_\pm=\OO$ are defined as $\rho_\pm\mape{\diag\left(g_+,g_-\right)}{g_\pm}$. Obviously, the irreducible Spin(9)-module $\OO^2$ decomposes under the action of $Spin(8)$ into two 8-dimensional irreducible components $S_+$ and $S_-$. In fact, the three modules $S_0$, $S_+$ and $S_-$ have more in common than just the same dimension. This is the content of the so-called {\it triality principle} for $Spin(8)$ (see References \cite{FH}, p. 312, \cite{harvey}, p. 275). Let us briefly explain one possible view of this phenomenon here.

First of all, the Lie group $Spin(8)$ has the following symmetric Dynkin diagram:
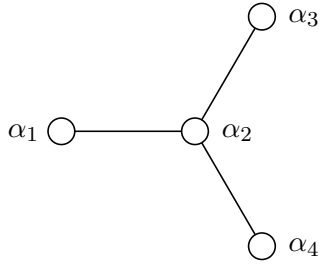
\begin{center}
\begin{picture}(120,100)  
\thicklines
\linethickness{0.6pt}
\put(20,50){\circle{10}}
\put(25,50){\line(1,0){40}}
\put(70,50){\circle{10}}
\put(72.5,54.33012702){\line(0.5,0.8660254040){20}}
\put(95,93.30127020){\circle{10}}
\put(72.5,45.66987298){\line(0.5,-0.8660254040){20}}
\put(95,6.6987298){\circle{10}}
\put(0,48){\makebox{$\alpha_{1}$}}
\put(80,48){\makebox{$\alpha_{2}$}}
\put(105,91.30127020){\makebox{$\alpha_{3}$}}
\put(105,4.6987298){\makebox{$\alpha_{4}$.}}
\end{picture}
\end{center}
In particular, there is a `rotational' symmetry preserving the root $\alpha_2$ and sending $\alpha_1,\alpha_3,\alpha_4$ to $\alpha_3,\alpha_4,\alpha_1$, respectively. This transformation induces clearly an automorphism of the corresponding Cartan subalgebra which then extends to an outer automorphism of the whole $\spin(8)$ (see Reference \cite{FH}, p. 338 and p. 498) and it lifts, finally, to an outer automorphism, say $\tau$, of $Spin(8)$. Using the inverse Cartan matrix, the fundamental weights $\lambda_i$ are expressed as follows:
\begin{align*}
\begin{pmatrix}\lambda_1\\\lambda_2\\\lambda_3\\\lambda_4\end{pmatrix}=\frac12\begin{pmatrix}2&2&1&1\\2&4&2&2\\1&2&2&1\\1&2&1&2\end{pmatrix}\begin{pmatrix}\alpha_1\\\alpha_2\\\alpha_3\\\alpha_4\end{pmatrix}
\end{align*}
Therefore, if $\rho$ is an irreducible representation of $Spin(8)$ with the highest weight $\sum_{i=1}^4 k_i\lambda_i$, for some $k_i\in\NN_0$, then $k_4\lambda_1+k_2\lambda_2+k_1\lambda_3+k_3\lambda_4$ is the highest weight of the (irreducible) representation $\rho\circ\tau$. It is well known that the fundamental weights are the highest weights of the vector representation $\rho_0$, the adjoint representation $\Adj$ and the positive and negative spin representations $\rho_\pm$ of $Spin(8)$, respectively. Thus, particularly, the {\it triality automorphism} $\tau$ rotates $\rho_0,\rho_+,\rho_-$ and fixes $\Adj$ in the following sense:
\begin{align*}
\rho_0\circ\tau\cong \rho_+,\quad \rho_+\circ\tau\cong \rho_-,\quad \rho_-\circ\tau\cong \rho_0\quad\text{and}\quad \Ad\circ\,\tau\cong\Ad.
\end{align*}

\section{Octonion-Valued Forms}
\label{s:OfV}
In this section, the notion of (real) alternating forms is extended by allowing them to take values in the octonions. To this end, recall that $\RR$ is naturally identified with $\linspan\{1\}\subset\OO$. Let $V$ be a $d$-dimensional real vector space. By $\RfVdeg{k}$, $0\leq k\leq d$, we denote the vector space of all $k$-forms on $V$. The exterior algebra of forms of all degrees is then $\RfV:=\bigoplus_{k=0}^d \RfVdeg{k}$.

\begin{definition}
Let $0\leq k\leq d$. We define
\begin{equation}
\bigwedge\!{}_\OO^k V^*:=\OO\otimes\bigwedge\!{}^k V^*.
\end{equation}
We call an element of $\OfVdeg{k}$ an {\it octonion-valued form of degree $k$ on $V$}. Further, we denote
\begin{equation}
\bigwedge\!{}_\OO^\bullet V^*:=\OO\otimes\bigwedge\!{}^\bullet V^*=\bigoplus_{k=0}^d \bigwedge\!{}_\OO^k V^*,
\end{equation}
the graded algebra equipped with the natural product
\begin{equation}
\label{eq:Owedge}
(u\otimes\varphi)\wedge(v\otimes\psi):=(uv)\otimes(\varphi\wedge\psi).
\end{equation}
\end{definition}

Notice that the real algebra $\OfV$ is neither associative nor alternating. Nonetheless, we find natural to denote the product \eqref{eq:Owedge} with the same wedge symbol, since it is an extension of the standard wedge product on $\RfV=\linspan\{1\}\otimes\RfV\subset\OfV$.

\begin{example}
\label{dxdy}
Let $V=\OO^2$. Then we define the {\it octonionic coordinate forms} $dx,dy\in\OfVdeg{1}$ as
\begin{equation}
dx\begin{pmatrix}u_1\\u_2\end{pmatrix}:=u_1\quad\text{and}\quad dy\begin{pmatrix}u_1\\u_2\end{pmatrix}:=u_2.
\end{equation}
If $\{e_0,\dots,e_7\}$ is an orthonormal basis of $\OO$ and $\{dx^0,\dots,dx^7,dy^0,\dots,dy^7\}$ is the corresponding canonical basis of $\RfVdeg{1}$, then
\begin{equation}
dx=\sum_{i=0}^7e_i\otimes dx^i\quad\text{and}\quad dy=\sum_{i=0}^7e_i\otimes dy^i.
\end{equation}
\end{example}

From now on, the following conventions will be adhered to. First of all, the tensor-product symbol will be omitted, i.e. $u\varphi:=u\otimes\varphi\in\OfV$, for the sake of brevity. Further, if $F\maps{\OO}{\OO}$ is an $\RR$-linear function, we define its extension to $\OfV$ by $F(u\varphi):=F(u)\varphi$. Examples of such functions we shall use are the involution, right/left multiplication by an octonion or the real-part operator.

To conclude this section, let us make three simple but important observations. First, assume $\alpha\in\OfVdeg{k}$ and $\beta\in\OfVdeg{l}$. It is straightforward to verify that
\begin{equation}
\label{eq:barwedge}
\overline{\alpha\wedge \beta}=(-1)^{kl}\,\b\beta\wedge\b\alpha.
\end{equation}
Second, for $u\varphi,v\psi\in\OfV$ one has
\begin{equation}
\label{eq:ipwedge}
\Real\left( u\varphi\wedge\b{v\psi}\right)=\Real(u\b v)\,\varphi\wedge\psi=\ip{u}{v}\,\varphi\wedge\psi.
\end{equation}
Third, since $\ip{u}{v}=\ip{\b u}{\b v}=\Real(\b u v)$, an immediate consequence of \eqref{eq:ipwedge} is that
\begin{equation}
\label{eq:realab}
\Real(\alpha\wedge\b\beta)=\Real(\b\alpha\wedge\beta)
\end{equation}
holds for any $\alpha,\beta\in\OfV$.

\section{Intermezzo: The Kraines 4-Form}

The formula \eqref{eq:kraines} for the Kraines form will be derived now. The 4-dimensional normed division algebra $\HH$ of {\it quaternions} can be viewed as a subalgebra of $\OO$, one possible choice is
\begin{align}
\HH:=\linspan\{1,e_1,e_2,e_4\}.
\end{align}
Similarly as for $\OO$, we define {\it quaternion-valued forms} on a $d$-dimensional real vector space $V$:
\begin{equation}
\bigwedge\!{}_\HH^k V^*:=\HH\otimes\bigwedge\!{}^k V^*\subset \bigwedge\!{}_\OO^k V^*\quad\text{and}\quad\bigwedge\!{}_\HH^\bullet V^*:=\HH\otimes\bigwedge\!{}^\bullet V^*=\bigoplus_{k=0}^d\bigwedge\!{}_\HH^k V^*\subset \bigwedge\!{}_\OO^\bullet V^*.
\end{equation}
Notice that, since $\OO$ is alternative and $e_4=e_1e_2$, associativity is recovered in the algebra $\HH$ and therefore the restriction of the wedge product on $\OfV$ to $\HfV$ is associative as well.

Let us recall the standard definition of the Kraines form \cite{kraines}. Assume $V=\HH^n$ and all summations being taken from $1$ to $n$ for the rest of this section. At first, we define
\begin{align}
\Omega_I(u,v):=\sum_i\ip{u_ie_1}{v_i},\quad \Omega_J(u,v):=\sum_i\ip{u_ie_2}{v_i},\quad \Omega_K(u,v):=\sum_i\ip{u_ie_4}{v_i},
\end{align}
for $u=(u_1,\dots, u_n),v=(v_1,\dots, v_n)\in\HH^n$. It is easily seen from \eqref{eq:RLuip} that $\Omega_I,\Omega_J,\Omega_K\in\RfVdeg{2}$. The {\it Kraines $4$-form} is then defined as
\begin{align}
\Omega:=\Omega_I\wedge\Omega_I+\Omega_J\wedge\Omega_J+\Omega_K\wedge\Omega_K.
\end{align}

Now we regard $\Omega$ as an element of $\HfVdeg{4}$ and express it in terms of the {\it quaternionic coordinate forms} $dw_1,\dots,dw_n\in\HfVdeg{1}$ that are naturally defined by
\begin{equation}
dw_i(u):=u_i, \quad1\leq i\leq n.
\end{equation}
For any $u,v\in\HH^n$ and $1\leq i\leq n$ we have
\begin{align*}
2\ip{u_ie_1}{v_i}=u_ie_1\b{v_i}-v_ie_1\b{u_i}=dw_i(u)e_1\b{dw_i(v)}-dw_i(v)e_1\b{dw_i(u)}=(dw_ie_1\wedge\b{ dw_i})(u,v),
\end{align*}
hence
\begin{align*}
2\Omega_I=\sum_iR_{e_1}(dw_i)\wedge\b{ dw_i}=\sum_idw_i\wedge L_{e_1}(\b{ dw_i}),
\end{align*}
and similarly for $\Omega_J$ and $\Omega_K$. The Kraines form therefore reads:
\begin{align*}
4\Omega&=\sum_{i,j}dw_i\wedge \left[ L_{e_1}(\b{dw_i})\wedge R_{e_1}(dw_j)+ L_{e_2}(\b{dw_i})\wedge R_{e_2}(dw_j)+ L_{e_4}(\b{dw_i})\wedge R_{e_4}(dw_j)\right]\wedge\b{dw_j}\\
&=\sum_{i,j}dw_i\wedge  (R_{e_1}L_{e_1}+R_{e_2}L_{e_2}+R_{e_4}L_{e_4})(\b{dw_i}\wedge dw_j)\wedge\b{dw_j}.
\end{align*}
Further, with help of the following lemma, $\Omega$ is rewritten more intrinsically, not depending on the particular choice of the basis $\{e_1,e_2,e_4\}$ of $\Imag\HH$.

\begin{lemma}
For any $\alpha\in\HfV$, we have
\begin{equation}
\label{eq:tvrzenicko}
(R_{e_1}L_{e_1}+R_{e_2}L_{e_2}+R_{e_4}L_{e_4})(\alpha)=-\alpha-2\b{\alpha}.
\end{equation}
\end{lemma}

\begin{proof}
Without loss of generality, we may assume $\alpha=u\varphi$ for some $u\in\HH$ and $\varphi\in\RfV$. Then since $\{1,e_1,e_2,e_4\}$ is an orthonormal basis of $\HH$, for any $u\in\HH$ we can write
\begin{align*}
\b u&=\ip{1}{\b u}1+\ip{e_1}{\b u}e_1+\ip{e_2}{\b u}e_2+\ip{e_4}{\b u}e_4 \\
&=\frac12\left[(u+\b u)+(e_1u+\b u\,\b{e_1})e_1+(e_2u+\b u\,\b{e_2})e_2+(e_4 u+\b u\,\b{e_4})e_4\right]\\
&=\frac12\left(u+4\b u+e_1ue_1+e_2ue_2+e_4ue_4\right),
\end{align*}
therefore $e_1ue_1+e_2ue_2+e_4ue_4=-u-2\b u$ and \eqref{eq:tvrzenicko} follows.
\end{proof}
\noindent
Thus, taking also \eqref{eq:barwedge} into account,
\begin{align*}
4\Omega&=\sum_{i,j}dw_i\wedge  (2\b{dw_j}\wedge dw_i- \b{dw_i}\wedge dw_j)\wedge\b{dw_j}.
\end{align*}
Now, from \eqref{eq:realab} we have
\begin{equation*}
\beta:=dw_i\wedge \b{dw_j}-dw_j\wedge \b{dw_i}=2\Real(dw_i\wedge \b{dw_j})=2\Real(\b{dw_i}\wedge dw_j)=\b{dw_i}\wedge dw_j-\b{dw_j}\wedge dw_i.
\end{equation*}
Obviously, $\beta\in\RfVdeg{2}$ and so it commutes with any element of $\HfV$. Then
\begin{align*}
(2\b{dw_j}\wedge dw_i- \b{dw_i}\wedge dw_j)\wedge\b{dw_j}&=2\b{dw_j}\wedge dw_j\wedge\b{dw_i}+2\b{dw_j}\wedge\beta-\b{dw_j}\wedge dw_i\wedge\b{dw_j}-\beta\wedge\b{dw_j} \\
&=\b{dw_j}\wedge dw_j\wedge\b{dw_i}+2\b{dw_j}\wedge\beta-\b{dw_j}\wedge \beta-\beta\wedge\b{dw_j} \\
&=\b{dw_j}\wedge dw_j\wedge\b{dw_i},
\end{align*}
and making use of the notation $\Omega_{ij}=dw_i\wedge \b{dw_j}$ established in the introduction, we finally have
\begin{align}
\label{eq:Omega}
\Omega=\frac14\sum_{i,j}dw_i\wedge \b{dw_j}\wedge dw_j\wedge\b{dw_i}=-\frac14\sum_{i,j}\Omega_{ij}\wedge\b{\Omega_{ij}}.
\end{align}

\section{The Spin(9)-Invariant 8-Form}

Finally, we use the notion of octonion-valued forms introduced in \S3, now in its full extent, in order to construct a non-trivial $Spin(9)$-invariant 8-form on the octonionic plane. It turns out that again, like in the case of the Kähler 2-form and the Kraines 4-form, the only ingredients needed for this construction are the coordinate 1-forms. As we observed in the previous section, the crucial part is to find the right way to glue these building blocks together. In this case, however, things are getting even more complicated `thanks' to non-associativity of $\OO$.

Let us begin with a technical lemma. For $\alpha_1,\dots,\alpha_4\in\OfV$, $V$ is again a general vector space, we define
\begin{equation}
\calF(\alpha_1,\alpha_2,\alpha_3,\alpha_4):=((\b{\alpha_1}\wedge\alpha_2)\wedge\b{\alpha_3})\wedge\alpha_4.
\end{equation}

\begin{lemma}
\label{lem:Ru}
For any $\alpha_1,\dots,\alpha_8\in\OfV$ and $u\in\OO$, $\norm{u}=1$, we have
\begin{equation}
\label{eq:RuPsi1}
\begin{split}
&\Real\calF(R_u\alpha_1,R_{\b u}\alpha_2,R_{\b u}\alpha_3,R_{\b u}\alpha_4)\wedge\b{\calF(R_{u}\alpha_5,R_{\b u}\alpha_6,R_{\b u}\alpha_7,R_{\b u}\alpha_8)} \\
&\qquad=\Real\calF(\alpha_1,\alpha_2,\alpha_3,\alpha_4)\wedge\b{\calF(\alpha_5,\alpha_6,\alpha_7,\alpha_8)},
\end{split}
\end{equation}
and
\begin{equation}
\label{eq:RuPsi2}
\begin{split}
&\Real\calF(R_u\alpha_1,R_{\b u}\alpha_2,R_{\b u}\alpha_3,R_{\b u}\alpha_4)\wedge\calF(R_{\b u}\alpha_5,R_{u}\alpha_6,R_{u}\alpha_7,R_{u}\alpha_8) \\
&\qquad=\Real\calF(\alpha_1,\alpha_2,\alpha_3,\alpha_4)\wedge\calF(\alpha_5,\alpha_6,\alpha_7,\alpha_8).
\end{split}
\end{equation}
\end{lemma}

\begin{proof}
Since the mapping $\calF$ is $\RR$-multilinear, we may without loss of any generality assume $\alpha_i=u_i\varphi_i$ for some $u_i\in\OO$ and $\varphi_i\in\RfV$, $1\leq i\leq 8$. Thus, taking the Moufang identities \eqref{eq:M1} and \eqref{eq:M3} and alternativity of $\OO$ into account, we can write
\begin{align*}
[[(\b u\,\b {u_1})(u_2\b u)](u\b{ u_3})](u_4\b u )&=[[\b u(\b{u_1}u_2)\b u](u\b{ u_3})](u_4\b u ) \\
&=[\b u[(\b{ u_1} u_2)(\b u(u\b{ u_3}))]](u_4\b u ) \\
&=[\b u((\b{ u_1}u_2)\b{ u_3})](u_4\b u ) \\
&=\b u [((\b{ u_1}u_2)\b{ u_3})u_4]\b u.
\end{align*}
Then, since $R_{\overline u},L_{\overline u}\in O(8)$ for $\norm{u}=1$,
\begin{align*}
\ip{[[(\b u\,\b {u_1})(u_2\b u)](u\b{ u_3})](u_4\b u )}{[[(\b u\,\b {u_5})(u_6\b u)](u\b{ u_7})](u_8\b u )}=\ip{((\b{ u_1}u_2)\b{u_3})u_4}{((\b{ u_5}u_6)\b{u_7})u_8}
\end{align*}
and \eqref{eq:RuPsi1} follows from \eqref{eq:ipwedge}. Similarly, we have
\begin{align*}
\b{[[(u\,\b {u_5})(u_6 u)](\b u\,\b{ u_7})](u_8 u )}=\b{u [((\b{ u_5}u_6)\b{ u_7})u_8]u}=\b u\,\b{[((\b{ u_5}u_6)\b{ u_7})u_8]}\,\b u,
\end{align*}
therefore
\begin{align*}
\ip{[[(\b u\,\b {u_1})(u_2\b u)](u\b{ u_3})](u_4\b u )}{\b{[[(u\,\b {u_5})(u_6 u)](\b u\,\b{ u_7})](u_8 u )}}=\ip{((\b{ u_1}u_2)\b{u_3})u_4}{\b{((\b{ u_5}u_6)\b{u_7})u_8}},
\end{align*}
and \eqref{eq:RuPsi2} then follows from \eqref{eq:ipwedge} rewritten in the form $\Real\left( w\varphi\wedge v\psi\right)=\ip{w}{\b v}\,\varphi\wedge\psi$.
\end{proof}

From now on we shall always assume $V=\OO^2$. In this case, one can consider the bi-grading $\RfV=\bigoplus_{k,l}\RfVdeg{k,l}$ with respect to $\OO^2=\OO\oplus\OO$. In agreement with the introduction we denote
\begin{align*}
\Psi_{40}&=\calF(dx,dx,dx,dx),\\
\Psi_{31}&=\calF(dy,dx,dx,dx),\\
\Psi_{13}&=\calF(dx,dy,dy,dy),\\
\Psi_{04}&=\calF(dy,dy,dy,dy).
\end{align*}
Notice that the definition of these octonionic 4-forms is independent of the choice of a basis for $\OO$, since the same is true for the 1-forms $dx$ and $dy$. Let $\det$ be the determinant on $\OO\cong\RR^8$ such that $\det(e_0,\dots,e_7)=1$ for the standard basis introduced in \S\ref{octonions}, and let us denote
\begin{align*}
\det\!{}_1&:=(dx)^*\det,\\
\det\!{}_2&:=(dy)^*\det.
\end{align*}
Here the forms $dx,dy\maps{\OO^2}{\OO}$ are regarded as the projections onto the first and second factor of $\OO^2$, respectively. In the following lemmas, we present two ways to construct the determinants from $dx$ and $dy$.

\begin{lemma}
\label{pro:det}
\begin{align}
\label{eq:det1}
\Psi_{40}\wedge\b{\Psi_{40}}&=8! \det\!{}_1,\\
\label{eq:det2}
\Psi_{04}\wedge\b{\Psi_{04}}&=8! \det\!{}_2.
\end{align}
\end{lemma}

\begin{proof}
Let $\{e_0,\dots,e_7\}$ be the standard orthonormal basis with $e_0:=1$. It then follows from \eqref{eq:RuRv} that $R_{\b {e_i}} R_{e_j}=-R_{\b {e_j}} R_{e_i}$ for any $0\leq i<j\leq 7$. By \eqref{eq:barwedge} we have $\Psi_{40}\wedge\b{\Psi_{40}}=\Real \Psi_{40}\wedge\b{\Psi_{40}}$. Thus, according to \eqref{eq:ipwedge},
\begin{align*}
\Psi_{40}\wedge\b{\Psi_{40}}&=\sum\ip{((\b{ e_{i_0}}e_{i_1})\b {e_{i_2}})e_{i_3}}{((\b{ e_{i_4}}e_{i_5})\b {e_{i_6}})e_{i_7}}\, dx^{i_0}\wedge\cdots\wedge dx^{i_7} \\
&=\sum\ip{R_{e_{i_4}} R_{\b{e_{i_5}}}R_{e_{{i_6}}}R_{\b{e_ {i_7}}} R_{e_{{i_3}}} R_{\b{e_{i_2}}} R_{e_{{i_1}}} R_{\b{ e_{i_0}}}(1)}{1} \,dx^{i_0}\wedge\cdots\wedge dx^{i_7},
\end{align*}
where the sum extends over all indices $0\leq i_0,\dots,i_7\leq 7$, but clearly only the terms with all indices distinct occur non-trivially. Since both factors in each term of the sum are totally skew-symmetric, we can write
\begin{align*}
\Psi_{40}\wedge\b{\Psi_{40}}&=8! \ip{R_{e_4} R_{\b{e_5}} R_{e_6} R_{\b{e_7}} R_{e_3} R_{\b{e_2}} R_{e_1} R_{\b{e_0}}(1)}{1}\,dx^{0}\wedge\cdots\wedge dx^{7} \\
&=8!\,dx^{0}\wedge\cdots\wedge dx^{7} \\
&=8! \det\!{}_1,
\end{align*}
to show \eqref{eq:det1}. Here we used $R_{e_4} R_{\b{e_5}} R_{e_6} R_{\b{e_7}} R_{e_3} R_{\b{e_2}} R_{e_1} R_{\b{e_0}}(1)=1$ which is easily verified by direct computation. The proof of \eqref{eq:det2} is completely analogous.
\end{proof}

\begin{remark}
\label{re:non-asoc}
Let $\{ e_0,\dots,e_7\}$ be the standard basis of $\OO$. Recall that $\b {e_i}=\pm e_i$, $e_ie_j=\pm e_je_i$ and that a product of two basis elements is, at most up to a sign, a member of the basis as well. Therefore we can write
\begin{align}
\label{eq:non-asoc}
e_i(e_je_k)=\varepsilon_1 e_i(e_ke_j)=\varepsilon_2 e_k(e_ie_j)=\varepsilon_3 (e_ie_j)e_k,
\end{align}
where the signs $\varepsilon_1,\varepsilon_2,\varepsilon_3=\pm1$ are in general independent of each other. The middle equality in \eqref{eq:non-asoc} follows, for $i\neq k$, from \eqref{eq:RuRv} and is trivial when $i=k$. All in all, a particular ordering of any product of the basis elements has effect on the sign at most. In this sense, the aforementioned relation $R_{e_4} R_{\b{e_5}} R_{e_6} R_{\b{e_7}} R_{e_3} R_{\b{e_2}} R_{e_1} R_{\b{e_0}}(1)=1$ implies
\begin{equation}
\label{eq:prod}
\prod_{k=0}^7e_k=\pm1.
\end{equation}
\end{remark}

\begin{lemma}
\label{lem:minus35}
\begin{align}
\Real\Psi_{40}\wedge\Psi_{40}&=-\frac35\, \Psi_{40}\wedge\b{\Psi_{40}}.
\end{align}
\end{lemma}
\begin{proof}
We shall work in the standard basis $\{e_0,\dots,e_7\}$ again. For any $0\leq i\leq 7$ we denote
\begin{align*}
\Psi_{40} ^i&=\sum ((\b{ e_{i_0}}e_{i_1})\b{ e_{i_2}})e_{i_3}\,dx^{i_0}\wedge dx^{i_1}\wedge dx^{i_2}\wedge dx^{i_3},
\end{align*}
if the sum runs over all indices with $ {e_{i_0}}e_{i_1}{ e_{i_2}}e_{i_3}=\pm e_i$ (see Remark \ref{re:non-asoc}). Again we make use of \eqref{eq:barwedge} and \eqref{eq:ipwedge} to write
\begin{align*}
\Psi_{40} ^i\wedge \b{\Psi_{40} ^i}&=\Real \Psi_{40} ^i\wedge \b{\Psi_{40} ^i} \\
&=\sum \ip{((\b {e_{i_0}}e_{i_1})\b {e_{i_2}})e_{i_3}}{((\b{ e_{i_4}}e_{i_5})\b{ e_{i_6}})e_{i_7}}\, dx^{i_0}\wedge\cdots \wedge dx^{i_7} \\
&=\sum\ip{R_{e_{i_4}} R_{\b{e_{i_5}}}R_{e_{{i_6}}}R_{\b{e_ {i_7}}} R_{e_{{i_3}}} R_{\b{e_{i_2}}} R_{e_{{i_1}}} R_{\b{ e_{i_0}}}(1)}{1} \,dx^{i_0}\wedge\cdots\wedge dx^{i_7}.
\end{align*}
Here the sums extend over all indices such that $ e_{i_0}e_{i_1} e_{i_2}e_{i_3}=\pm e_i$ (and $ e_{i_4}e_{i_5} e_{i_6}e_{i_7}=\pm e_i$ but this is redundant since the inner product would be zero otherwise). Like in the previous proof, due to skew-symmetry we further have
\begin{align*}
\Psi_{40} ^i\wedge \b{\Psi_{40} ^i}=n_i (4!)^2 \det\!{}_1,
\end{align*}
where $n_i$ denotes the number of combinations of four distinct indices $0\leq i_0,i_1,i_2,i_3\leq 7$ such that $e_{i_0}e_{i_1} e_{i_2}e_{i_3}=\pm e_i$. We claim that, among all ${8\choose4}=70$ combinations of four distinct indices, 14 of these products equal $\pm1$. To see this, assume $e_{i_0}e_{i_1} e_{i_2}e_{i_3}=\pm e_{i_4}e_{i_5} e_{i_6}e_{i_7}=\pm1$ for $i_0,\dots,i_7$ being all distinct. If one of the indices $i_0,i_1,i_2,i_3$, say $i_0$, is zero, then the others are non-zero and $e_{i_1}e_{i_2}e_{i_3}=\pm1$, hence $e_{i_3}=\pm e_{i_1}e_{i_2}$. There are precisely 7 distinct sets $\{i_1,i_2,i_3\}$ satisfying this, corresponding to the 7 cases in \eqref{eq:Omul}. Symmetrically, the other 7 combinations occur when $0\in\{i_4,i_5,i_6,i_7\}$. Together, $n_0=14$, therefore $\sum_{i=1}^7n_i=70-14=56$ and
\begin{align*}
\Real\Psi_{40}\wedge\Psi_{40}=\Psi_{40} ^0\wedge \b{\Psi_{40} ^0}-\sum_{i=1}^7  \Psi_{40}^i\wedge\b{\Psi_{40}^i}=\left(n_0-\sum_{i=1}^7 n_i\right) (4!)^2 \det\!{}_1=-\frac35\, \Psi_{40}\wedge\b{\Psi_{40}} .
\end{align*}
\end{proof}

Let us prove an auxiliary assertion from the representation theory of $Spin(8)$ before we finally proceed to the proof of the main result. Recall that the notation is kept from \S\ref{prelim}.

\begin{lemma}
\label{lem:5}
$\dim\left[\bigwedge^8\left(S_+\oplus S_-\right)^*\right]^{Spin(8)}=5$.
\end{lemma}

\begin{proof}
Regarded as a representation space of $Spin(8)$,
\begin{align*}
\bigwedge\!{}^8 (S_+\oplus S_-)^*\cong \bigwedge\!{}^8 (S_+\oplus S_-)=\bigoplus_{k=0}^8 \bigwedge\!{}^k S_+\otimes\bigwedge\!{}^{8-k} S_-
\end{align*}
and thus
\begin{align*}
d:=\dim\left[\bigwedge\!{}^8 (S_+\oplus S_-)^*\right]^{Spin(8)}=\sum_{k=0}^8\dim\left[ \bigwedge\!{}^k S_+\otimes\bigwedge\!{}^{8-k} S_-\right]^{Spin(8)}.
\end{align*}
Since $\diag(\id,-\id)\in Spin(8)$, the terms of the sum corresponding to odd values of $k$ are trivial and, because $\dim S_\pm=8$ and so $\bigwedge\!{}^{8-k} S_\pm\cong \bigwedge\!{}^k S_\pm$, we in fact have $d=2d_0+2d_2+d_4$, where
\begin{align*}
d_k:=\dim\left[ \bigwedge\!{}^k S_+\otimes\bigwedge\!{}^{k} S_-\right]^{Spin(8)}.
\end{align*}

Let $\Gamma_\mu$ be an irreducible $Spin(8)$-module of the highest weight $\mu$. In particular, we have $S_0=\Gamma_{\lambda_1}$, $S_+=\Gamma_{\lambda_3}$ and $S_-=\Gamma_{\lambda_4}$. Trivially, $\bigwedge^0S_0=\Gamma_0$. It is also well known that $\bigwedge^2S_0=\Gamma_{\lambda_2}$ is the adjoint representation and that $\bigwedge^4S_0=\Gamma_{2\lambda_3}\oplus\Gamma_{2\lambda_4}$  (see \cite{FH}, \S19.2). Applying the triality principle, we further obtain $\bigwedge^0S_+=\bigwedge^0S_-=\Gamma_0$, $\bigwedge^2S_+=\bigwedge^2S_-=\Gamma_{\lambda_2}$ and $\bigwedge^4S_+=\Gamma_{2\lambda_4}\oplus\Gamma_{2\lambda_1}$ while $\bigwedge^4S_-=\Gamma_{2\lambda_1}\oplus\Gamma_{2\lambda_3}$. Counting the same factors in the decompositions of exterior powers of $S_+$ and $S_-$, we finally conclude $d_0=d_2=d_4=1$ and thus $d=5$.
\end{proof}

\begin{proof}[Proof of Theorem \ref{thm1}]
Let us denote
\begin{align*}
\Psi_{80}&:=\Psi_{40}\wedge\b{\Psi_{40}},\\
\Psi_{62}&:=\Psi_{31}\wedge\b{\Psi_{31}},\\
\Psi_{44}&:=-\frac{5}{6}\left(\Psi_{31}\wedge\Psi_{13}+\b{\Psi_{13}}\wedge\b{\Psi_{31}}\right)=-\frac53\Real\Psi_{31}\wedge\Psi_{13},\\
\Psi_{26}&:=\Psi_{13}\wedge\b{\Psi_{13}},\\
\Psi_{08}&:=\Psi_{04}\wedge\b{\Psi_{04}}.
\end{align*}
We shall show that
\begin{equation}
\label{eq:psi8=}
\Psi_8=\Psi_{80}+4\,\Psi_{62}+6\,\Psi_{44}+4\,\Psi_{26}+\Psi_{08}
\end{equation}
is a non-trivial $Spin(9)$-invariant real $8$-form.

First, all summands of $\Psi_8$ belong clearly to $\OfVdeg{8}$. In fact, they are all real. This is seen from \eqref{eq:barwedge}, taking into account that $\Psi_{40},\Psi_{31},\Psi_{13},\Psi_{04}$ are of even degree.  Hence $\Psi_8\in\RfVdeg{8}$.

Second, we prove that each summand is separately $Spin(8)$-invariant. To this end, assume $0\leq j<k\leq 7$ and let us express
\begin{equation*}
g_{jk}(t)^*=g_{jk}(t)^{-1}=g_{jk}(-t)=\diag\left(R_{e_j}\circ R_{\b {u_{jk}(-t)}},R_{\b{ e_j}}\circ R_{u_{jk}(-t)}\right),
\end{equation*}
as $g_{jk}(t)\in SO(16)$. Clearly, $\norm{e_j}=\norm{u_{jk}(-t)}=1$, for any $t\in\RR$. Then, since $R_u\in SO(8)$ for $\norm{u}=1$ and the $n$-dimensional determinant is $SO(n)$-invariant, the forms $\Psi_{80}$ and $\Psi_{08}$ are $Spin(8)$-invariant according to Lemma \ref{pro:det}. The invariance of the rest follows immediately from Lemma \ref{lem:Ru} which is applied twice: first for $u=u_{jk}(-t)$ and second for $u=e_j$.

Finally, for $t\in\RR$, let us abbreviate $c:=\cos(t)$, $s:=\sin(t)$ and
\begin{equation*}
g^*:=g_{08}(t)^*=g_{08}(-t)=\begin{pmatrix}R_c&R_s\\ -R_s &R_c\end{pmatrix}.
\end{equation*}
Let, further, $\calP\maps{\RfVdeg{8}}{\RfVdeg{8,0}}$ denote the natural projection. Since $g^* dx= cdx +sdy$ and $g^* dy= -sdx +c dy$, with help of Lemma \ref{lem:minus35} it is not difficult to see that, for all $k,l$ we consider,
\begin{align*}
\calP\left(g^*\Psi_{kl}\right)&=c^ks^l\,\Psi_{80}.
\end{align*}
In particular, this shows that all the five forms $\Psi_{kl}$ are non-trivial. Further, since $\Psi_{kl}\in\RfVdeg{k,l}$, they are also linearly independent and thus, according to Lemma \ref{lem:5}, span $\left(\RfVdeg{8}\right)^{Spin(8)}$. Because $Spin(8)\subset Spin(9)$, we have $\Psi\in\left(\bigwedge\!{}^8 V^*\right)^{Spin(8)}$, and therefore there are constants $\kappa:=(\kappa_0,\dots,\kappa_4)\in\RR^5$ such that $\Psi_\kappa:=\sum_{i=0}^4\kappa_i\Psi_{8-2i,2i}$ is $Spin(9)$-invariant. In order to fix $\kappa$, we impose the condition of invariance under $g^*$. Namely, in particular,
\begin{equation*}
\calP\left(g^*\Psi_\kappa\right)=\sum_{i=0}^4\kappa_ic^{8-2i}s^{2i}\,\Psi_{80}
\end{equation*}
must equal to
\begin{equation*}
\calP\left(\Psi_\kappa\right)=\kappa_0\Psi_{80}=\kappa_0(c^2+s^2)^4\,\Psi_{80}.
\end{equation*}
It is easily seen that the solution of
\begin{equation*}
\sum_{i=0}^4\kappa_ic^{8-2i}s^{2i}= \kappa_0(c^2+s^2)^4
\end{equation*}
equals uniquely, up to scaling by $\kappa_0$, to the binomial coefficients of the fourth-power expansion. In particular, for $\kappa_0=1$ we have $\Psi_8=\Psi_\kappa$ which completes the proof.
\end{proof}

\section*{Appendix}

We now express the form $\Psi_8$ explicitly in terms of the dual basis $\{dx^0,\dots,dx^7,dy^0,\dots,dy^7\}$ of $\bigwedge^1(\OO^2)^*$ corresponding to the standard basis $\{e_0,\dots,e_7\}$ of $\OO$. Although the computations we perform to this end are slightly more technical, they are based on very elementary algebraic properties of the octonions. Basically, we just use the formula \eqref{eq:prod} together with the rule $R_{\b {e_i}} R_{e_j}=-R_{\b {e_j}} R_{e_i}$, whenever $i\neq j$, following easily from \eqref{eq:RuRv}. We shall also keep the notation from the proof of the main theorem and omit the wedge product symbol for the sake of brevity.

\subsection*{The Parts $\Psi_{80}$ and $\Psi_{08}$}
As already shown in Lemma \ref{pro:det}, both these parts consist of one element each. Namely, $\Psi_{80}=8! \,dx^0\cdots dx^7$ and $\Psi_{08}=8! \,dy^0\cdots dy^7$.

\subsection*{The Parts $\Psi_{62}$ and $\Psi_{26}$}

According to \eqref{eq:ipwedge} we have
\begin{align*}
4\,\Psi_{62}&=4\sum \ip{R_{e_{i_4}} R_{\b{e_{i_5}}}R_{e_{{i_6}}}R_{\b{e_ {i_7}}} R_{e_{{i_3}}} R_{\b{e_{i_2}}} R_{e_{{i_1}}} R_{\b{ e_{i_0}}}(1)}{1} \,dy^{i_0}dx^{i_1}dx^{i_2}dx^{i_3}dy^{i_4}dx^{i_5}dx^{i_6}dx^{i_7},
\end{align*}
or, after reordering the canonical 1-forms,
\begin{align*}
4\,\Psi_{62}&=-4\sum \ip{R_{e_{i_7}} R_{\b{e_{i_3}}}R_{e_{{i_4}}}R_{\b{e_ {i_5}}} R_{e_{{i_2}}} R_{\b{e_{i_1}}} R_{e_{{i_0}}} R_{\b{ e_{i_6}}}(1)}{1} \,dx^{i_0}\cdots dx^{i_5}dy^{i_6}dy^{i_7}.
\end{align*}
Clearly, a general term
\begin{align}
\label{eq:gt62}
-4 \ip{R_{e_{i_7}} R_{\b{e_{i_3}}}R_{e_{{i_4}}}R_{\b{e_ {i_5}}} R_{e_{{i_2}}} R_{\b{e_{i_1}}} R_{e_{{i_0}}} R_{\b{ e_{i_6}}}(1)}{1} \,dx^{i_0}\cdots dx^{i_5}dy^{i_6}dy^{i_7}
\end{align}
of this sum is possibly non-trivial only if $\#\{i_0,\dots,i_5\}=6$ and $\#\{i_6,i_7\}=2$. Hence, there are just three eventualities for $\#\{i_0,\dots,i_7\}$: 6, 7 or 8. Further necessary condition on non-triviality of \eqref{eq:gt62} is obviously, in a sense of Remark \ref{re:non-asoc},
\begin{equation}
\label{eq:prodi}
\prod_{k=0}^7e_{i_k}=\pm1.
\end{equation}

First, suppose $\#\{i_0,\dots,i_7\}=8$. This means all indices in \eqref{eq:gt62} are distinct and the inner product there is thus totally skew-symmetric. Therefore, there are ${8\choose2}=28$ distinct terms of this kind, each corresponding to a different set $\{i_6,i_7\}$, all with coefficients $\pm4\cdot2!\cdot6!=\pm8\cdot 6!$.

Second, let $\#\{i_0,\dots,i_7\}=7$, i.e. let precisely two indices coincide in \eqref{eq:gt62}. Then \eqref{eq:prodi} requires that the product of six distinct basis vectors equals $\pm1$. According to \eqref{eq:prod}, this would however mean that the product of the two remaining (and distinct) basis elements is also $\pm1$, which is impossible. We conclude, therefore, that there is no non-trivial term of this kind.

Finally, suppose $\#\{i_0,\dots,i_7\}=6$, i.e. $\{i_6,i_7\}\subset\{i_0,\dots,i_5\}$. In particular, $i_6$ agrees with precisely one element in $\{i_0,\dots,i_5\}$ and thus, after commuting the operator $R_{\b{e_{i_6}}}$ leftwards, \eqref{eq:gt62} takes the form
\begin{align*}
+4\,\ip{R_{e_{i_7}} R_{\b{e_{i_6}}}R_{e_{{i_3}}}R_{\b{e_ {i_4}}} R_{e_{{i_5}}} R_{\b{e_{i_2}}} R_{e_{{i_1}}} R_{\b{ e_{i_0}}}(1)}{1} \,dx^{i_0}\cdots dx^{i_5}dy^{i_6}dy^{i_7}
\end{align*}
that is totally skew-symmetric in $i_6,i_7$ and in $i_0,\dots,i_5$, respectively. The inner product is, however, non-zero precisely when the product of the basis elements of indices $\{i_0,\cdots,i_5\}\backslash\{i_6,i_7\}$ is $\pm1$. So, as shown during the proof of Lemma \ref{lem:minus35}, there are $14$ possibilities for the set $\{i_0,\cdots,i_5\}\backslash\{i_6,i_7\}$ and to each of them there are ${4\choose 2}=6$ choices of $\{i_6,i_7\}$. Therefore, there are $6\cdot14=84$ terms of this kind, each with prefactor $\pm4\cdot2!\cdot6!=\pm8\cdot 6!$.

The case of $\Psi_{26}$ is completely analogous.

\subsection*{The Part $\Psi_{44}$}
Now we have
\begin{align*}
6\,\Psi_{44}&=-10\sum  \ip{((\b {e_{i_0}}e_{i_1})\b {e_{i_2}})e_{i_3}}{\b{((\b{ e_{i_4}}e_{i_5})\b{ e_{i_6}})e_{i_7}}} \,dy^{i_0}dx^{i_1}dx^{i_2}dx^{i_3}dx^{i_4}dy^{i_5}dy^{i_6}dy^{i_7},
\end{align*}
so, after reordering, a general term takes the form
\begin{align}
\label{eq:gt44}
-10  \ip{((\b {e_{i_4}}e_{i_0})\b {e_{i_1}})e_{i_2}}{\b{((\b{ e_{i_3}}e_{i_5})\b{ e_{i_6}})e_{i_7}}} \,dx^{i_0}\cdots dx^{i_3}dy^{i_4}\cdots dy^{i_7},
\end{align}
that is only non-trivial if $\#\{i_0,\dots,i_3\}=\#\{i_4,\dots,i_7\}=4$, i.e. $4\leq\#\{i_0,\dots,i_7\}\leq8$. Due to the higher complexity of this case, we introduce the following product of indices: $(i,j)\mapsto{ij}$, where $ij$ is the (unique) element of $\{0,\dots,7\}$ such that $e_{ij}=\pm e_ie_j$. Such a product is clearly commutative as well as associative (see Remark \ref{re:non-asoc}). The condition \eqref{eq:prodi}, which of course still applies, translates in this language as
\begin{equation}
\label{eq:prodin}
\prod_{k=0}^7i_k=0.
\end{equation}

Let $\#\{i_0,\dots,i_7\}=8$, i.e. $\{i_0,\dots,i_3\}\cap\{i_4,\dots,i_7\}=\emptyset$. We shall distinguish two cases here. First, suppose $i_0i_1i_2i_3=0$. Then \eqref{eq:prodin} is only fulfilled if $i_4i_5i_6i_7=0$ too, i.e. if $i_5i_6i_7=i_4$. Since $i_3\neq i_4$, one has $i_5i_6i_7\neq i_3$ and thus $i_3i_5i_6i_7\neq0$. Therefore $\b{((\b{ e_{i_3}}e_{i_5})\b{ e_{i_6}})e_{i_7}}=-((\b{ e_{i_3}}e_{i_5})\b{ e_{i_6}})e_{i_7}$ and so \eqref{eq:gt44} takes the form
\begin{align*}
&+10\,\ip{((\b {e_{i_4}}e_{i_0})\b {e_{i_1}})e_{i_2}}{((\b{ e_{i_3}}e_{i_5})\b{ e_{i_6}})e_{i_7}} \,dx^{i_0}\cdots dx^{i_3}dy^{i_4}\cdots dy^{i_7}\\
&\qquad=+10\ip{R_{e_{i_3}} R_{\b{e_{i_5}}}R_{e_{{i_6}}}R_{\b{e_ {i_7}}} R_{e_{{i_2}}} R_{\b{e_{i_1}}} R_{e_{{i_0}}} R_{\b{ e_{i_4}}}(1)}{1}\,dx^{i_0}\cdots dx^{i_3}dy^{i_4}\cdots dy^{i_7},
\end{align*}
that is again totally skew-symmetric and thus the coefficient is $\pm10\cdot4!\cdot4!=\pm8\cdot 6!$. We have already shown above that there exist 14 distinct sets $\{i_0,\dots,i_3\}$, such that $i_0i_1i_2i_3=0$, and there are therefore 14 terms of this kind. Second, if $i_0i_1i_2i_3\neq0$ then, by \eqref{eq:prodin}, also $i_4i_5i_6i_7\neq0$ and thus $i_5i_6i_7\neq i_4$. If, for instance, $i_5i_6i_7= i_5$, then $i_6=i_7$, which is impossible. Similarly one shows that $i_5i_6i_7\neq i_6$ and  $i_5i_6i_7\neq i_7$. It is therefore necessary that $i_5i_6i_7\in\{i_0,i_1,i_2,i_3\}$. If $i_5i_6i_7=i_3$, we have $\b{((\b{ e_{i_3}}e_{i_5})\b{ e_{i_6}})e_{i_7}}=((\b{ e_{i_3}}e_{i_5})\b{ e_{i_6}})e_{i_7}$ and \eqref{eq:gt44} reads
\begin{align*}
-10\,\ip{((\b {e_{i_4}}e_{i_0})\b {e_{i_1}})e_{i_2}}{((\b{ e_{i_3}}e_{i_5})\b{ e_{i_6}})e_{i_7}} \,dx^{i_0}\cdots dx^{i_3}dy^{i_4}\cdots dy^{i_7}.
\end{align*}
In the three other cases $i_5i_6i_7\in\{i_0,i_1,i_2\}$, $\b{((\b{ e_{i_3}}e_{i_5})\b{ e_{i_6}})e_{i_7}}=-((\b{ e_{i_3}}e_{i_5})\b{ e_{i_6}})e_{i_7}$ and \eqref{eq:gt44} equals
\begin{align*}
+10\,\ip{((\b {e_{i_4}}e_{i_0})\b {e_{i_1}})e_{i_2}}{((\b{ e_{i_3}}e_{i_5})\b{ e_{i_6}})e_{i_7}} \,dx^{i_0}\cdots dx^{i_3}dy^{i_4}\cdots dy^{i_7}.
\end{align*}
Hence, the coefficient in front of such a term is $\pm\left(\frac{-1+3}{4}\right)\cdot 10\cdot 4!\cdot 4!=\pm4\cdot6!$. As discussed in the proof of Lemma \ref{lem:minus35}, there are 56 sets $\{i_0,\dots,i_3\}$ with  $i_0i_1i_2i_3\neq 0$ and so is the number of the corresponding terms.

If $\#\{i_0,\dots,i_7\}=7$, then \eqref{eq:prodin} could never be fulfil from exactly the same reason as in the case of $\Psi_{62}$. There is, hence, no such term again.

\begin{table}
\centering
\extrarowsep=2pt
$\begin{tabu}{|c||c|c|c|c|c|c|c|c|c|c|c|c|c|c|c|c|}
\hline
i_3 & j_0&j_0&j_0&j_0&j_1&j_1&j_1&j_1&j_2&j_2&j_2&j_2&j_3&j_3&j_3&j_3
\\\hline
i_4 & j_2&j_3&j_4&j_5& j_2&j_3&j_4&j_5& j_2&j_3&j_4&j_5& j_2&j_3&j_4&j_5
\\\hline
\e_1 & -&-&-&-&-&-&-&-&+&-&-&-&-&+&-&-
\\\hline
\e_2 & +&+&-&-&+&+&-&-&+&-&+&+&-&+&+&+
\\\hline
\e_1\e_2 &-&-&+&+&-&-&+&+&+&+&-&-&+&+&-&-
\\\hline
\end{tabu}$
\vspace{12pt}
\caption{The signs $\e_1\e_2$ in the case $j_0j_1j_2j_3=0$}
\label{t1}
\end{table}

\begin{table}
\centering
\extrarowsep=2pt
$\begin{tabu}{|c||c|c|c|c|c|c|c|c|c|c|c|c|c|c|c|c|}
\hline
i_3 & j_0&j_0&j_0&j_0&j_1&j_1&j_1&j_1&j_2&j_2&j_2&j_2&j_3&j_3&j_3&j_3
\\\hline
i_4 & j_2&j_3&j_4&j_5& j_2&j_3&j_4&j_5& j_2&j_3&j_4&j_5& j_2&j_3&j_4&j_5
\\\hline
\e_1 & -&-&-&+&-&-&+&-&-&-&-&-&-&-&-&-
\\\hline
\e_2 & +&+&-&-&+&+&-&-&+&-&+&+&-&+&+&+
\\\hline
\e_1\e_2 &-&-&+&-&-&-&-&+&-&+&-&-&+&-&-&-
\\\hline
\end{tabu}$
\vspace{12pt}
\caption{The signs $\e_1\e_2$ in the case $j_0j_1j_2j_3\neq0$}
\label{t2}
\end{table}

Let $\#\{i_0,\dots,i_7\}=6$, and denote $\{j_0,\dots,j_3\}:=\{i_0,\dots,i_3\}$ and $\{j_2,\dots,j_5\}:=\{i_4,\dots,i_7\}$. According to \eqref{eq:prodin}, we may assume $j_0j_1j_4j_5=0$, i.e. $j_0j_1=j_4j_5$. Let $\e_1,\e_2=\pm1$ be such that
\begin{align*}
\b{((\b{ e_{i_3}}e_{i_5})\b{ e_{i_6}})e_{i_7}}&=\e_1((\b{ e_{i_3}}e_{i_5})\b{ e_{i_6}})e_{i_7},\\
 R_{e_{i_3}} R_{\b{e_{i_5}}}R_{e_{{i_6}}}R_{\b{e_ {i_7}}} R_{e_{{i_2}}} R_{\b{e_{i_1}}} R_{e_{{i_0}}} R_{\b{ e_{i_4}}}&=\e_2 R_{e_{i_5}} R_{\b{e_{i_6}}}R_{e_{{i_7}}}R_{\b{e_ {i_4}}} R_{e_{{i_3}}} R_{\b{e_{i_2}}} R_{e_{{i_1}}} R_{\b{ e_{i_0}}}.
\end{align*}
Using this notation, a general term \eqref{eq:gt44} takes the form
\begin{align*}
6\,\Psi_{44}=-10\, \e_1\e_2 \ip{ R_{e_{i_5}} R_{\b{e_{i_6}}}R_{e_{{i_7}}}R_{\b{e_ {i_4}}} R_{e_{{i_3}}} R_{\b{e_{i_2}}} R_{e_{{i_1}}} R_{\b{ e_{i_0}}}(1)}{1} \,dx^{i_0}\cdots dx^{i_3}dy^{i_4}\cdots dy^{i_7}.
\end{align*}
In what follows, we shall discuss how the sign $\e_1\e_2$ alternates for different positions of $i_3$ within $\{j_0,\dots,j_3\}$ and of $i_4$ within $\{j_2,\dots,j_5\}$. Let us distinguish two separate cases. First, assume $j_0j_1j_2j_3=0$ or equivalently $j_0j_1=j_2j_3$. Then, if $i_3=j_0$ and $i_4=j_2$, for instance, one has $\{i_5,i_6,i_7\}=\{j_3,j_4,j_5\}$ and $\{i_0,i_1,i_2\}=\{j_1,j_2,j_3\}$. Therefore,
\begin{align*}
i_3i_5i_6i_7=j_0j_3j_4j_5=j_0j_3j_2j_3=j_0j_2\neq 0,
\end{align*}
meaning $((\b{ e_{i_3}}e_{i_5})\b{ e_{i_6}})e_{i_7}\neq\pm1$ and thus $\e_1=-1$. Further, since $i_3\notin \{i_5,i_6,i_7\}$, $i_4\notin \{i_0,i_1,i_2\}$ and $i_3\neq i_4$, respectively, we can write
\begin{align*}
 R_{e_{i_3}} R_{\b{e_{i_5}}}R_{e_{{i_6}}}R_{\b{e_ {i_7}}} R_{e_{{i_2}}} R_{\b{e_{i_1}}} R_{e_{{i_0}}} R_{\b{ e_{i_4}}}&= -R_{e_{i_5}} R_{\b{e_{i_6}}}R_{e_{{i_7}}}R_{\b{e_ {i_3}}} R_{e_{{i_2}}} R_{\b{e_{i_1}}} R_{e_{{i_0}}} R_{\b{ e_{i_4}}} \\
&= -R_{e_{i_5}} R_{\b{e_{i_6}}}R_{e_{{i_7}}}R_{\b{e_ {i_3}}} R_{e_{{i_4}}} R_{\b{e_{i_2}}} R_{e_{{i_1}}} R_{\b{ e_{i_0}}} \\
&= R_{e_{i_5}} R_{\b{e_{i_6}}}R_{e_{{i_7}}}R_{\b{e_ {i_4}}} R_{e_{{i_3}}} R_{\b{e_{i_2}}} R_{e_{{i_1}}} R_{\b{ e_{i_0}}},
\end{align*}
and thus $\e_2=+1$. The signs corresponding to the other positions of $i_3$ and $i_4$ are computed analogically and summarised in Table \ref{t1}. One can observe from the table that $\e_1\e_2$ equals +1 in precisely 8 cases and -1 in the 8 others, from which we conclude that the corresponding term is trivial in the end. We may thus assume $j_0j_1j_2j_3\neq0$. Then it is easily seen that the eight indices $j_0$, $j_1$, $j_2$, $j_3$, $j_0j_1j_2$, $j_0j_1j_3$, $j_0j_2j_3$ and $j_1j_2j_3$ are all distinct. Therefore, $j_4$ and $j_5$ must be among the last four ones. The requirement $j_0j_1j_4j_5=0$ however chooses the last two ones. Without loss of generality, we thus have $j_4=j_0j_2j_3$ and $j_5=j_1j_2j_3$. Now we investigate the behaviour of the sign $\e_1\e_2$ again, taking into account that $j_0j_4=j_1j_5=j_2j_3$. The results are captured in Table \ref{t2}. Clearly, $\e_2$ stays the same as in the case $j_0j_1j_2j_3=0$ but $\e_1$ alternates so that $\e_1\e_2$ is positive only in 4 cases and negative otherwise. This means that the corresponding term appears with the coefficient $\pm\left(\frac{-12+4}{16}\right)\cdot 10\cdot 4!\cdot 4!=\pm4\cdot6!$. Regarding the number of such terms, there are $56$ options for $\{j_0,\dots,j_3\}$, $j_0j_1j_2j_3\neq0$, and for each of them there are ${4\choose 2}=6$ possible partitions into $\{j_0,j_1\}$ and $\{j_2,j_3\}$. Since $\{j_4,j_5\}$ is then uniquely determined, there are altogether $56\cdot6=336$ terms of this kind.

Further, suppose that $\#\{i_0,\dots,,i_7\}=5$, i.e. that $\{i_0,\dots,i_3\}\cap\{i_4,\dots,i_7\}$ contains precisely three indices. \eqref{eq:prodin} requires that the product of the two (distinct) elements of $\{i_0,\dots,i_7\}$ that do not belong to this intersection is $0$. This is again impossible and thus there are no terms here either.

Finally, let $\#\{i_0,\dots,,i_7\}=4$, i.e. $\{i_0,\dots,i_3\}=\{i_4,\dots,i_7\}$. First, assume $i_0i_1i_2i_3=0$. If $i_3=i_4$, then it is easily seen that $\e_1=\e_2=1$. If $i_3\neq i_4$, then $\e_1=\e_2=-1$. In any case $\e_1\e_2=1$ and so 14 these terms have all coefficients $\pm10\cdot 4!\cdot4!=\pm8\cdot 6!$. Second, if $i_0i_1i_2i_3\neq 0$, then $\e_1=-1$ regardless the relation between $i_3$ and $i_4$. Since $\e_2$ does not change from the previous case, for any $i_3$ we have $\e_1\e_2=-1$ if $i_4=i_3$ and $\e_1\e_2=1$ in the three other cases of $i_4\neq i_3$. Altogether, the prefactors of these 56 terms are $\pm\left(\frac{-1+3}{4}\right)\cdot 10\cdot 4!\cdot 4!=\pm4\cdot6!$.

\subsection*{Summary}

All in all, the expression of $\Psi_8$ in the standard basis possesses 702 non-trivial terms. They are summarised in Table \ref{t3}. Each block of the table corresponds to one summand in \eqref{eq:psi8=}. Each row of the table stands for a particular class of terms of $-\frac{1}{4\cdot 6!}\Psi_8$. A general term of the class is stated in the second column and the class is further specified in the third column. In the first column, the coefficient standing in front of the terms from the respective class is given. Let us remark that we scaled the form $\Psi_8$ by $-\frac{1}{4\cdot 6!}$ in order to adhere to the conventions of \cite{PP}. Notice that the signs of the coefficients can be explicitly determined directly from the aforedescribed construction. Finally, the number of non-trivial terms within each class is given in the fourth column. Throughout the table, we assume $i_k\neq i_l$ if $k\neq l$. Recall also that the product of indices is taken in the following sense: $e_{ij}=\pm e_ie_j$.

\vspace{12pt}

\begin{table}[H]
\centering
\extrarowsep=2pt
$\begin{tabu}{|c|c|c|c|}
\hline
\text{Coefficient}&\text{Basis vector}&\text{Specification}&\text{Number}\\
\hline\hline
-14&dx^0\cdots dx^7&-&1 \\\hline\hline
\pm2&dx^{i_0}\cdots dx^{i_5}dy^{i_6}dy^{i_7}&i_0<\cdots<i_5;\,i_6<i_7&28
\\\hline
\pm 2&dx^{i_0}\cdots dx^{i_5}dy^{i_4}dy^{i_5}&i_0<\cdots<i_3;\,i_4<i_5;\,i_0i_1i_2i_3=0&84
\\\hline\hline
\pm2&dx^{i_0}\cdots dx^{i_3}dy^{i_4}\cdots dy^{i_7}&i_0<\cdots<i_3;\,i_4<\dots<i_7;\,i_0i_1i_2i_3=0&14
\\\hline
\pm1&dx^{i_0}\cdots dx^{i_3}dy^{i_4}\cdots dy^{i_7}&i_0<\cdots<i_3;\,i_4<\dots<i_7;\,i_0i_1i_2i_3\neq0&56
\\\hline
\multirow{2}{*}{$\pm1$}&\multirow{2}{*}{$dx^{i_0}\cdots dx^{i_3}dy^{i_2}\cdots dy^{i_5}$}&i_0<i_1;\,i_2<i_3;\,i_0i_1i_2i_3\neq0;&\multirow{2}{*}{$336$}\\
&&i_4=i_0i_2i_3;\,i_5=i_1i_2i_3&
\\\hline
\pm1&dx^{i_0}\cdots dx^{i_3}dy^{i_0}\cdots dy^{i_3}&i_0<\cdots<i_3;\,i_0i_1i_2i_3\neq 0&56
\\\hline
\pm2&dx^{i_0}\cdots dx^{i_3}dy^{i_0}\cdots dy^{i_3}&i_0<\cdots<i_3;\,i_0i_1i_2i_3=0&14
\\\hline\hline
\pm2&dx^{i_0}dx^{i_1}dy^{i_0}\cdots dy^{i_5}&i_0<i_1;\,i_2<\cdots<i_5;\,i_2i_3i_4i_5=0&84
\\\hline
\pm2&dx^{i_0}dx^{i_1}dy^{i_2}\cdots dy^{i_7}&i_0<i_1;\,i_2<\cdots<i_7&28
\\\hline\hline
-14&dy^0\cdots dy^7&-&1
\\\hline
\end{tabu}$
\caption{Explicit expression of the form $-\frac{1}{4\cdot 6!}\Psi_8$ in the standard basis}
\label{t3}
\end{table}

\newpage


\end{document}